\documentclass[11pt,english]{amsart}

\usepackage{amsmath,leftidx,amsthm}
\usepackage[all]{xy}
\usepackage{xspace,amsthm}
\usepackage[psamsfonts]{amssymb}
\usepackage[latin1]{inputenc}
\usepackage{graphicx,color,mathrsfs}
\usepackage{hyperref,fancyhdr,dutchcal,appendix}
\usepackage{xcolor}

\newtheorem*{example}{\bf Example}

\newtheorem*{remark}{\bf Remark}

\newtheorem{theorem}{\bf Theorem}[section]
\newtheorem{proposition}[theorem]{\bf Proposition}

\newtheorem{lemma}[theorem]{\bf Lemma}
\newtheorem{corollary}[theorem]{\bf Corollary}

\def\C{{\mathbb C}}
\def\N{{\mathbb N}}
\def\R{{\mathbb R}}

\def\B{\mathbb{B}}

\def\supp{\textup{supp}}

\title{Lebesgue points of functions in the complex Sobolev space}

\author{Gabriel Vigny}
\address{LAMFA, Universit\'e de Picardie Jules Verne, 33 rue Saint-Leu, 80039 AMIENS Cedex 1, FRANCE}
\email{gabriel.vigny@u-picardie.fr}

\author{Duc-Viet Vu}
\address{Department of Mathematics and Computer Sciences, University of Cologne, Germany}
\email{dvu@uni-koeln.de}

\thanks{Both authors are partially supported by the ANR-DFG grant QuaSiDy, grant no ANR-21-CE40-0016.}

\date{\today}
\begin{document}

\begin{abstract} Let $\varphi$ be a function in the complex Sobolev space $W^*(U)$, where $U$ is an open subset in $\C^k$.  We show that the complement of the set of Lebesgue points of $\varphi$ is pluripolar.  The key ingredient in our approach is to show that $|\varphi|^\alpha $ for $\alpha \in [1,2)$ is locally bounded from above by a plurisubharmonic function.
\end{abstract}

\maketitle

\medskip

\noindent
{\bf Classification AMS 2020}:  32Uxx, 37Fxx.

\medskip


\section{Introduction} \label{introduction}

Let $U$ be a bounded open set in $\C^k$ endowed with the standard Euclidean form $\omega:= \frac{1}{2}dd^c |z|^2$ (recall that $d=\partial + \bar{\partial}$ and $d^c:=\frac{i}{2 \pi} (\bar \partial -\partial)$). Let $W^{1,2}(U)$ be the Sobolev space of function $\varphi$ in $L^2(U)$ such that $\nabla \varphi \in L^2(U)$. Let $W^*(U)$ be the subspace of $W^{1,2}(U)$ consisting of $\varphi \in  W^{1,2}(U)$ such that  there exists a positive closed current $T$ of bidegree $(1,1)$ and of finite mass (i.e., $\int_U T \wedge \omega^{k-1} < \infty$) on $U$ such that  
\begin{equation}\label{bound_by_current}
d \varphi \wedge d^c \varphi \leq T.
\end{equation}

When $k=1$, then $W^*(U)=W^{1,2}(U)$ since $d \varphi \wedge d^c \varphi$  is already a positive measure. The space $W^*(U)$, which we call \emph{the Dinh-Sibony-Sobolev space} (in the literature, it is also called the complex Sobolev space), was introduced by Dinh and Sibony in \cite{DinhSibonyW} to show the exponential decay of correlations for $C^\alpha$ observables in complex dynamics. Indeed, the importance of $W^*(U)$ lies in the fact that, by its very definition, it takes into account the complex structure (which $W^{1,2}$ does not) and it is stable under birational transformation. Since then, the space $W^*(U)$ has provided many applications, e.g. in complex dynamics \cite{Vignydecay, Vu_equilibrium, bianchi2022equilibrium}, in random matrix theory \cite{DinhKaufmannWu}, in the study of Monge-Amp\`ere equations \cite{DinhKoloNguyen,Vu-log-diameter}.  

Recall that $A \subset U$ is said to be a pluripolar set  if there exists a plurisubharmonic (psh) function $\varphi$ on $U$ such that 
$A \subset \{\varphi =  -\infty\}$. 
For every Borel set $K$ in $U$, the (Bedford-Taylor) \emph{capacity} of $K$ in $U$ (see \cite{BedfordTaylor}) is given by
$$\mathrm{Cap}(K,U):= \sup \bigg\{ \int_K (dd^c v)^k:  0 \le v \le 1, \, \text{$v$ is  psh on $U$} \bigg\},$$
and  Borel sets of zero capacity in $U$  are exactly pluripolar sets. 

 Following Dinh and Sibony, in \cite{VignyW}, the first author showed that $W^*(U)$ is actually a Banach space endowed with the norm
\[ \|\varphi\|_*= \int_U |\varphi| \omega^{k} +  \inf \left(\int_U T \wedge \omega^{k-1}\right)^{1/2} \]    
where the infimum is taken over all the positive closed current of bidegree $(1,1)$ satisfying \eqref{bound_by_current}. This norm allows to define a functional capacity which is comparable to Bedford-Taylor capacity  and the elements of $W^*(U)$ are continuous outside open subsets of $U$ of arbitrarily small capacity. 
Furthermore, in \cite{DinhMarVu}, it was proved that the standard regularization  by convolution of an element $\varphi$ in $W^*(U)$, after extracting a subsequence if necessary, converges pointwise to $\varphi$ outside some pluripolar set (see \cite[Theorem 2.10]{DinhMarVu}). Precisely, let $\mathrm{Leb}$ be the Lebesgue measure on $\C^k$.  If $\varphi \in W^*(U)$ and $\chi$ is a cut-off radial function with compact support in $U$ with $\int_{\C^k} \chi d \mathrm{Leb}=1$ and 
$$\varphi_\epsilon(x):= \epsilon^{-2k}\int_{\C^k} \varphi(x-y) \chi(y/\epsilon) d \mathrm{Leb},$$
then 
there is a subsequence  $(\varphi_{\epsilon_n})_n \subset (\varphi_\epsilon)_\epsilon$  which converges pointwise to a Borel function $\varphi'$ outside some pluripolar set in $U$. The function $\varphi'$ is equal to $\varphi$ in $L^2(U)$ and is called a \emph{representative} of $\varphi$. Two representatives differ only on a pluripolar set. \emph{From now on, when we speak of  a function in $W^*(U)$, we implicitly identify it with one of its representatives.}


Our main result in this paper strengthens the above-mentioned result by showing that the complement of the set of Lebesgue points of $\varphi$ is indeed pluripolar. In particular, it implies that $\varphi_\epsilon $ converges pointwise to $\varphi$ outside some pluripolar set as $\epsilon \to 0$.

Let  $B(x,r)$ denote the ball of radius $r>0$ with center $x$ in $\C^k$.
A point $x \in U$ is a \emph{Lebesgue point} for a function $ \varphi \in L^1_{loc}(U)$  if one has 
\begin{align}\label{ine-lebes}
\lim_{\epsilon \to 0}  \frac{1}{\mathrm{Leb}(B(0,\epsilon))}\int_{B(0,\epsilon)} |\varphi(x+ y)-\varphi(x)| d \mathrm{Leb}(y)\to 0,
\end{align} 
It is classical that the  convergence  (\ref{ine-lebes}) holds for almost every $x$. In other words, the complement of the set of Lebesgue points of $\varphi$ is of zero Lebesgue measure. 


\begin{theorem}\label{th:Lebesgues} Let $\varphi \in W^*(U)$ and  let $\varphi'$ be a representative of $\varphi$. Then,
the complement of the set of Lebesgue points of $\varphi'$ is pluripolar. 
	In particular, $\varphi_\epsilon$ converges pointwise to $\varphi'$ as $\epsilon \to 0$ outside some pluripolar set. 
\end{theorem}

 Let $X$ be a complex manifold. The notion of Lebesgue points extends to points in $X$. 
We define $W^*_{loc}(X)$ to be the subset of $L^2_{loc}(X)$ consisting of functions $\varphi$ so that for every $x \in X$, there exists a local chart $U$ around $x$ satisfying that $\varphi \in W^*(U)$.  As a direct consequence of  Theorem  \ref{th:Lebesgues}, one obtains

\begin{corollary} \label{cor-localW*}  Let $X$ be a complex manifold. Let $\varphi \in W^*_{loc}(X)$ and let $\varphi'$ be a representative of $\varphi$. Then the complement of the set of Lebesgue points of $\varphi'$ is locally pluripolar.
\end{corollary}

Endow $X$ with a smooth Riemannian metric. Let $B_X(x,r)$ be the ball of radius $r$ centered at $x \in X$.  Let $\mu_0$ be a smooth volume form on $X$.
By Corollary  \ref{cor-localW*}, for every $\varphi \in W^*_{loc}(X)$,  we define a Borel function $\tilde{\varphi}: X \to \R$ as follows:  
$$\tilde{\varphi}(x):= \lim_{\epsilon \to 0} \frac{1}{\mu_0(B_X(x,\epsilon))} \int_{B_X(x,\epsilon)} \varphi d\mu_0$$
if the limit exists (and is a finite number), otherwise we simply put $\tilde{\varphi}(x):=0$. One sees that $\tilde{\varphi}(x)= \varphi(x)$ if $x$ is a Lebesgue point of $\varphi$, and $\tilde{\varphi}$ is a representative of $\varphi$. We call  $\tilde{\varphi}$  \emph{a canonical representative} of  $\varphi$.    In practice we don't distinguish $\tilde{\varphi}$ with $\varphi$.   It is useful to recall at this point that if $X$ is compact, then every locally pluripolar set on $X$ is actually (globally) pluripolar set; see \cite{DS_tm,GZ,Vu_pluripolar}.

Let us now comment on the proof of Theorem \ref{th:Lebesgues}. As explained in \cite[Corollary 25]{VignyW}, because of the continuity outside sets of arbitrarily small capacity and using the plurifine topology, the desired assertion of Theorem \ref{th:Lebesgues} would be a consequence of the fact that $\varphi$ can be controlled, locally, by a plurisubharmonic function in the sense that 
\begin{equation}\label{usual_bound}
\forall V\Subset U, \ \exists u \ \mathrm{psh \ on \  V}, u \leq \varphi\leq -u \, \text{ on }\, V.
\end{equation}
This fact is well-known in the dimension one. Let us recall briefly arguments for a proof of \eqref{usual_bound}  when $k=1$ (hence $W^*(U)=W^{1,2}(U)$). Assume that $\varphi$ is non-positive with 
compact support in $U$ (one can always reduce to that case) so that we can actually work in $W_0^{1,2}(U)$ endowed with the norm $\|\varphi\|^2 := \int_U  d \varphi \wedge d^c\varphi $. Consider
\[  C_\varphi:= \{ \phi \in W^*(U), \phi \leq \varphi \leq 0 \}.\]
Then  $C_\varphi$ is a closed convex set in $W_0^{1,2}(U)$ so it admits a minimal element by Riesz projection theorem: a 
function $\phi$ such that $\int_U   d \phi \wedge d^c \phi $ is minimal 
in $C_\varphi$. Then, $\phi$ is subharmonic because for every negative test function 
$\theta$ with compact support in $U$ and $t>0$, we have that  $\phi + t \theta \in C_\varphi$. Hence
$$\int_U d (\phi + t \theta) \wedge d^c (\phi + t \theta) \geq \int_U d 
\phi \wedge d^c \phi.$$
 Using this and  Stokes' formula and letting $t \to 0$ implies $\int_U \theta dd^c \phi  \leq 0$. It follows that $dd^c \phi$ is  a positive measure, i.e, $\phi$ is subharmonic.

When $k>1$, it is not clear how to generalize such a strategy because $W^*(U)$ is not a Hilbert space anymore (it is not even reflexive \cite[Corollary 8]{VignyW}). The key ingredient in the proof of Theorem \ref{th:Lebesgues} is the following result which is of independent interest.  

\begin{theorem}\label{tm:main?} Let $\B$ be the unit ball in $\C^k$. 
	Let $\varphi \in W^*(\B)$  with $\|\varphi\|_* \le 1$
	and $\alpha \in [1,2)$. Then for every compact $K \subset \B$, there exist  a constant $C>0$ and  a psh function $u$ on $\B$ such that 
\begin{align}\label{ine-chanphiboiumualpha}	
	 |\varphi|^{\alpha }\leq -u
	\end{align}
	on    $K$ possibly outside some pluripolar set and $\|u\|_{L^1(K)} \le C$. 
\end{theorem}

We recall again that in (\ref{ine-chanphiboiumualpha}) we implicitly identify $\varphi$ with a representative so that the inequality is meant to hold outside a pluripolar set.
Our proof of Theorem  \ref{tm:main?} relies  on a sort of $L^m$-bound against Monge-Amp\`ere measures of bounded potentials for functions in $W^*(\B)$ and a ``plurisubharmonicify" argument to construct plurisubharmonic functions through a given Borel function.
 Constructions of similar types were used  in pluripotential theory; e.g., see the proof of Josefson's theorem \cite[Theorem 4.7.4]{Klimek}.

Using exponential integrability of psh functions and Theorem \ref{tm:main?}, we see that for every compact $K \Subset \B$ and for every $\alpha \in [1,2)$,  there exist constants $c_1>0,c_2>0$ such that for  every $\varphi \in W^*(\B)$  with $\|\varphi\|_* \le 1$ there holds
$$\int_K e^{c_1 |\varphi|^\alpha} d \mathrm{Leb} \le c_2.$$
This is a weak version of Moser-Trudinger type inequality proved in \cite[Theorem 1.2]{DinhMarVu} where it was even showed that the above inequality does hold for $\alpha=2$ by a different method. 
We don't know if Theorem \ref{tm:main?} holds for $\alpha=2$.  Investigating that case is interesting because if the answer is positive, this gives a new proof of the above-mentioned Moser-Trudinger type inequality. Nevertheless, Theorem \ref{tm:main?} is  no longer true if $\alpha>2$, see Example \ref{ex-alpha2} below.  We also obtain a global version of Theorem \ref{tm:main?}, see Theorem \ref{th-tm:main?global} below for details. 

In the next section we prove Theorem \ref{tm:main?}. The proof of Theorem \ref{th:Lebesgues} is given in Section \ref{sec-mainthoerporof}. 
\\

\noindent
\textbf{Acknowledgment.} We thank the anonymous referee for a very careful reading and many suggestions  which improved considerably the presentation of the paper.

\section{Plurisubharmonic upper bound for functions in $W^*(\B)$} \label{sec-pluribound}

In this section we prove Theorem \ref{tm:main?}. \emph{We first do it for $\varphi \in W^*(\B)$ non-negative, and the general case will be treated at the end of this section.} Let $\varphi \in W^*(\B)$ be a non-negative function with $\|\varphi\|_* \le 1$ and let $\psi$ be a negative psh function on $\B$ so that 
$$ d \varphi \wedge d^c \varphi \le dd^c \psi.$$ 
By \cite[Theorem 2.10]{DinhMarVu}, we know that there exists a Borel function $\varphi': \B \to \R$ such that $\varphi'= \varphi$ almost everywhere and if $\varphi_\epsilon$ is the standard regularization of $\varphi$ by convolution, then $\varphi_{\epsilon}$ converges to $\varphi'$ in capacity, and there exists a subsequence $(\varphi_{\epsilon_j})_j \subset (\varphi_\epsilon)_\epsilon$ such that $\varphi_{\epsilon_j}$ converges pointwise to $\varphi'$ outside some pluripolar set. Such a  function $\varphi'$ is called a representative of $\varphi.$  We recall also that 
\begin{align}\label{ine-chinhquyhoa}
d \varphi_\epsilon \wedge d^c\varphi_\epsilon \le dd^c \psi_\epsilon,
\end{align}
where $\psi_\epsilon$ is the standard regularization of $\psi$.

\subsection{Energy estimates} 
We note that it was proved in \cite[Proposition 6]{VignyW} that $\varphi$ belongs to BMO space and consequently, by results in \cite{John_nirenberg}, one gets $\int_K e^{c \varphi} \omega^k \le A$ for some positive constants $c, A$ and for every $\varphi$ with $\|\varphi\|_* \le 1$. In this paper, we use the following direct consequence of this exponential estimate (it can also be deduced from the Moser-Trudinger inequality proved in \cite{DinhMarVu}).

\begin{corollary}\label{cor-Lmbound} Let $K \Subset \B$ and $m\in \N$.  There exists a constant $C_1>0$ so that  $$\int_K |\varphi|^m \omega^k \le C_1$$
for every $\varphi \in W^*(\B)$ with $\|\varphi\|_* \le 1$.
\end{corollary}

Let $n \in \N$ and  $\psi_n:= \max\{\psi, -n\}$.  Consider the function
\[ h_n:= 1 + \psi_n/ n \cdot \]
We have that $0\leq h_n\leq 1$ is a psh function with $ h_n=0$ on $\{\psi \leq  -n\}$.  Let $T_n:= dd^c h_n^2/2$, it is a positive closed $(1,1)$-current. Observe that $T_n$ vanishes on the open set $\{ \psi <  -n\}$. 

\begin{lemma}\label{le-obser} Assume that $\varphi$ and $\psi$ are smooth. The following inequalities hold:

(i) $d h_n \wedge d^c h_n \leq  T_n$ and $ h_n  dd^c h_n \leq  T_n$

(ii) $ d \varphi \wedge d^c \varphi \leq   dd^c \psi_n$
on $\{h_n >0\}$,

(iii) $ d \varphi \wedge d^c \varphi \wedge T_n \leq   dd^c \psi_{n+1} \wedge T_n.$
\end{lemma}

\proof  Direct computations show $T_n= h_n dd^c h_n+ d h_n \wedge d^c h_n$. Hence, we get 
$$T_n \ge d h_n \wedge d^c h_n, \quad T_n \ge h_n dd^c h_n $$
 because $h_n$ is a non-negative psh function. Thus (i) follows.

 Since $\psi$ is continuous, we have $dd^c \psi_n= dd^c \psi$ on the open set $\{\psi > -n\}= \{h_n> 0\}$. Hence (ii) follows. 
 
  We now prove (iii).  Observe that $d \varphi \wedge d^c \varphi \wedge T_n=0$ on $\{h_{n+1}=0\}= \{\psi \le -n-1\}$ because $T_n=0$ on the open set $\{\psi < -n\}$. By (ii),
$$d \varphi \wedge d^c \varphi \wedge T_n \le dd^c \psi_{n+1} \wedge T_n$$
on $\{h_{n+1}>0\}$. Consequently (iii) follows.  
\endproof

 For every $m\in \N$,   $K \Subset \B$, for $0 \le p \le k$, we define 
$$I_{n,m,p,K}:= \sup_{v_1,\ldots, v_p} \int_{K} h_n^2 \varphi^{2m}  dd^c v_1 
\wedge \cdots \wedge dd^c v_p \wedge \omega^{k-p}$$
and for $0 \le p \le k-1$,
$$J_{n,m,p,K}:= \sup_{v_1,\ldots, v_p} \int_{K} \varphi^{2m}  dd^c v_1 
\wedge \cdots \wedge dd^c v_{p} \wedge T_n \wedge \omega^{k-1-p},$$
where the supremum  in the  two definitions above is taken over all psh functions $v_j$ on $\B$ with $0 \le v_j \le 1$. Note that in the definitions of $I_{n,m,p,K}$ and $J_{n,m,p,K}$, we have fixed a representative of $\varphi$. 
These definitions are independent of choices of representatives.  The quantity $I_{n,m,p,K}$ can be considered as a sort of energy for $\varphi$.

 In what follows, we use $\lesssim$ or $\gtrsim$ to denote $\le$ or $\ge$ respectively modulo a multiplicative constant independent of $n$ and $\varphi$.
  
\begin{lemma} \label{le-estimateJp=0} Let $\varphi \in W^*(\B)$ be a non-negative function with $\|\varphi\|_* \le 1$ as above.  Then, there exists a constant $c= c_{m,K}$ independent of  $n$ and $\varphi$, such that 	
 $$J_{n,m,0,K} \leq c n^m$$
 for every $n$. 
\end{lemma}

\proof Recall that for every constant $M \ge 0$, the function $\min\{\varphi, M\}$ also belongs to $W^*(\B)$ with $\|\min\{\varphi, M\}\|_* \le 1$, and $$d \min(\varphi, M) \wedge d^c \min(\varphi, M) \le dd^c \psi.$$
Thus by considering $\min\{\varphi, M\}$ instead of $\varphi$ (and using Lebesgue's monotone convergence theorem with $M \to \infty$), we can assume that $\varphi$ is bounded. By standard regularization and Lebesgue's dominated convergence theorem, we can further assume that $\varphi$ and $\psi$ are smooth (see (\ref{ine-chinhquyhoa})).

Let $\chi$ be a smooth cut-off function such that $\chi \equiv 1$ on an open neighborhood of $K$, and $\chi$ is compactly supported on $\B$ with $0 \le \chi \le 1$.  In order to prove the desired assertion, it suffices to bound from above
$$J_{n,m,0,\chi}:=\int_\B \chi^2  \varphi^{2m}  T_n \wedge \omega^{k-1}.$$
Since $T_n= dd^c h_n^2/2$, Stokes' formula gives
\begin{align*}
J_{n,m,0,\chi} &= \int_\B \varphi^{2m} \chi   d \chi \wedge d^c h^2_n \wedge \omega^{k-1}+  m \int_\B \chi^2 \varphi^{2m-1}  d \varphi \wedge d^c h^2_n \wedge \omega^{k-1}\\
&= 2 \int_\B \varphi^{2m} \chi  h_n d \chi \wedge d^c h_n \wedge \omega^{k-1}+  2m \int_\B \chi^2 \varphi^{2m-1} h_n d \varphi \wedge d^c h_n \wedge \omega^{k-1}.
\end{align*}
Let $A_1, A_2$ be the first and second terms in the right-hand side of the last equality respectively. Since $\chi$ is smooth, we obtain $d \chi \wedge d^c \chi \lesssim \omega$.  By this, Cauchy-Schwarz inequality applied to $d \chi \wedge (\chi d^c h_n)$ and Lemma \ref{le-obser}  one gets
\begin{align*}
A_1^2 &\lesssim \int_\B \chi^2 \varphi^{2m } d h_n \wedge d^c h_n \wedge \omega^{k-1} \int_\B \varphi^{2m} d \chi \wedge d^c \chi \wedge \omega^{k-1}\\
& \lesssim  \int_\B \chi^2 \varphi^{2m } T_n \wedge \omega^{k-1} \int_{\supp \chi} \varphi^{2m} \omega^{k}\\
& \lesssim  J_{n,m,0,\chi}
\end{align*}
by  Corollary \ref{cor-Lmbound}. We treat $A_2$ similarly. We have
\begin{align*}
A_2^2 &\lesssim  \int_\B \chi^2 \varphi^{2m} h_n^2 d \varphi \wedge d^c \varphi \wedge \omega^{k-1} \int_\B \chi^2 \varphi^{2m-2} d h_n \wedge d^c h_n \wedge \omega^{k-1}\\
&\lesssim  \int_\B \chi^2 \varphi^{2m} h_n^2 dd^c \psi_n \wedge \omega^{k-1} \int_\B \chi^2 \varphi^{2m-2} T_n \wedge \omega^{k-1}\\
&\lesssim  n J_{n,m-1,0,\chi} J_{n,m,0, \chi}
\end{align*}
because of Lemma \ref{le-obser} again and $h_n^2 dd^c \psi_n=  n h_n^2 dd^c h_n \leq n h_n T_n \leq nT_n$. Consequently, we get
$$J_{n,m,0,\chi} \le C (J_{n,m,0,\chi})^{1/2}+ C(n J_{n,m-1,0,\chi} J_{n,m,0,\chi})^{1/2}$$
for some constant $C$ independent of $n,\varphi$. We infer
$$J_{n,m,0,\chi} \le C^2(1+ n J_{n,m-1,0,\chi}).$$
Applying the last inequality inductively for $m-1, \ldots, 1$ instead of $m$, we obtain 
$$J_{n,m,0,\chi}\lesssim n^m.$$
This finishes the proof.
\endproof

\begin{lemma} \label{le-estimateJ} Let $\varphi \in W^*(\B)$ be a non-negative function with $\|\varphi\|_* \le 1$ as above.  Then, there exists a constant $c= c_{m,K}$ independent of $n$ and $\varphi$, such that 
 $$J_{n,m,p,K} \leq c n^m$$
 for every $n$. 
\end{lemma}

\proof We prove the lemma by induction on $p$. If $p=0$, the desired assertion is Lemma \ref{le-estimateJp=0}. We assume now that it is true for all $p'$ with $p' \le p-1$. 

As in the proof of Lemma \ref{le-estimateJp=0}, without loss of generality, we can assume that $\varphi$ and $\psi$ are smooth.   Let $\chi$ be a smooth cut-off function as in the proof of Lemma \ref{le-estimateJp=0}.   In order to prove the desired assertion, it suffices to bound from above
$$J_{n,m,p,\chi}:= \sup_{v_1,\ldots, v_{p}} \int_\B \chi^2  \varphi^{2m}  dd^c v_1 
\wedge \cdots \wedge dd^c v_{p}\wedge T_n \wedge \omega^{k-1-p}.$$
We check that 
$$J_{n,m,p,\chi} \lesssim n^m$$
by induction on $m$ ($p$ now fixed). When $m=0$, this is obvious. Assume that it is true for $m' \le m-1$. Let $R:= dd^c v_2 \wedge\dots \wedge dd^c v_p \wedge T_n \wedge \omega^{k-p-1}$. We note that $R$ depends on $n$. However, to ease the notation, we don't explicitly write the dependence on $n$ here.     By Stokes' formula one gets
\begin{align*}
&\int_{\mathbb{B}} \chi^2 \varphi^{2m}  dd^c v_1 
\wedge R = \\
&	-2\int_{
		\mathbb{B}}  \varphi^{2m} \chi d \chi\wedge   d^c v_1 
	\wedge R-2m\int_{
		\mathbb{B}} \chi^2 \varphi^{2m-1}d\varphi \wedge d^c v_1 
	\wedge R.
\end{align*}
Denote by $Q_1,Q_2$ the first and second terms in the right-hand side of the last equality respectively. By the Cauchy-Schwarz  inequality applied to $d \chi \wedge (\chi d^c v_1)$ and the fact that $\chi$ is smooth with $d \chi\wedge   d^c \chi  \lesssim 1_{\supp \chi } \omega$ and $0 \le v_1 \le 1$, we obtain 
\begin{align} \label{ine-danhgiaQ1}
Q_1^2  &\le 4\int_{
 		\mathbb{B}}  \varphi^{2m} d \chi\wedge   d^c \chi 
	\wedge R    \int_{ 
 		\mathbb{B}} \chi^2 \varphi^{2m} d v_1 \wedge   d^c v_1 
	\wedge R\\
\nonumber	
	& \lesssim \int_{
 		\supp \chi}  \varphi^{2m} 
	R \wedge \omega    \int_{
 		\mathbb{B}}\chi^2 \varphi^{2m} dd^c v_1^2 	\wedge R\\
	\nonumber
	& \lesssim J_{n,m,p,\chi} J_{n,m,p-1,\supp \chi} \lesssim J_{n,m,p,\chi} n^m 
	\end{align}
	by the induction hypothesis on $p$, where in the second inequality of (\ref{ine-danhgiaQ1}), we used the estimate: 
$$dd^c v_1^2= 2 v_1 dd^c v_1+ 2 d v_1\wedge d^c v_1 \ge 2 d v_1 \wedge d^c v_1.
$$	
	 We estimate $Q_2$ similarly. By Cauchy-Schwarz inequality applied to $(\varphi^{m-1} d \varphi) \wedge (\varphi^m d^c v_1)$ and Lemma \ref{le-obser} (iii),
 \begin{align*}
Q_2^2  	& \lesssim    \int_{
 		\mathbb{B}} \chi^2 \varphi^{2m}dv_1\wedge d^c v_1 
 	\wedge R \int_{
 		\mathbb{B}} \chi^2 \varphi^{2(m-1)}d\varphi\wedge d^c \varphi
 	\wedge R \\
& \lesssim J_{n,m,p,\chi} \int_{
 		\mathbb{B}} \chi^2 \varphi^{2(m-1)} dd^c \psi_{n+1} \wedge R\\ 	
 	& \lesssim (n+1) \, J_{n,m,p,\chi} \int_{
 		\mathbb{B}} \chi^2 \varphi^{2(m-1)}  dd^c h_{n+1} \wedge R.
 \end{align*}
By induction hypothesis on $m$, one gets 
 $$\int_{
 	\mathbb{B}} \chi^2 \varphi^{2(m-1)}  dd^c h_{n+1} \wedge
R \lesssim n^{m-1}.$$
 It follows that 
 $$Q_2^2 \lesssim n^m J_{n,m,p,\chi}. $$
 This coupled with (\ref{ine-danhgiaQ1}) yields that 
 $$J_{n,m,p, \chi}^2 \lesssim n^m J_{n,m,p,\chi}.$$
 Hence $J_{n,m,p,\chi} \lesssim n^m$.  
 This finishes the proof.
\endproof

\begin{lemma}	 \label{le-estimateI} Let $\varphi \in W^*(\B)$ be a non-negative function with $\|\varphi\|_* \le 1$ as above.  Then, there exists a constant $c= c_{m,K}$ independent of $n$ and $\varphi$ such that 
 $$I_{n,m,p,K} \leq c n^m$$
 for every $n$. 
\end{lemma}

\begin{proof} We argue similarly as in the proof of Lemma  \ref{le-estimateJ}, by induction on $p$. If $p=0$, the desired assertion follows from Corollary \ref{cor-Lmbound}. We assume now that it is true for every $p'$ with $p' \le p-1$. 

As in the proof of Lemma \ref{le-estimateJp=0}, without loss of generality, we can assume that $\varphi$ and $\psi$ are smooth.   Let $\chi$ be a smooth cut-off function as in the proof of Lemma \ref{le-estimateJp=0}.  In order to prove the desired assertion, it suffices to bound from above
$$I_{n,m,p,\chi}:= \sup_{v_1,\ldots, v_{p}} \int_\B \chi^2 h_n^2 \varphi^{2m}  dd^c v_1 
\wedge \cdots \wedge dd^c v_{p} \wedge \omega^{k-p}.$$
We check that 
$$I_{n,m,p,\chi} \lesssim n^m$$
by induction on $m$ ($p$ now fixed). When $m=0$, this is obvious because $h_n$ is bounded, $\supp \chi \Subset \B$ and the desired inequality follows from the standard Chern-Levine-Nirenberg inequality for Monge-Amp\`ere operators. Assume that it is true for $m' \le m-1$. Let $R:= dd^c v_2 \wedge \cdots dd^c v_p \wedge \omega^{k-p}$.  By Stokes' formula, one gets 
\begin{align*}
 \int_{
	\mathbb{B}} \chi^2 h_n^2 \varphi^{2m}  dd^c v_1 
\wedge R  &= 	 -\int_{
	\mathbb{B}}  h_n^2 \varphi^{2m} 2\chi d\chi \wedge  d^c v_1 
\wedge R-   \quad \quad \int_{
	\mathbb{B}} \chi^2 h_n^2  2m \varphi^{2m-1} d\varphi  \wedge d^c v_1 
\wedge R-\\
& \quad  \int_{
	\mathbb{B}} \chi^2 \varphi^{2m}  2h_n d h_n \wedge d^c v_1 
\wedge R.
\end{align*}
Denote by $P_1,P_2,P_3$ the first, second and third terms in the right-hand side of the last equality. Arguing as in the estimation of $Q_1,Q_2$ in the proof of Lemma  \ref{le-estimateJ}, one obtains 
\begin{align}\label{ine-P1}
P_1^2 \le 4\int_{
 		\mathbb{B}}  h_n^2 \varphi^{2m}  d\chi \wedge  d^c \chi 
\wedge R \int_{
 		\mathbb{B}}\chi^2 h_n^2 \varphi^{2m}  d v_1 \wedge  d^c v_1 
\wedge R \lesssim I_{n,m,p-1,\supp \chi} I_{n,m,p,\chi} \lesssim n^m I_{n,m,p,\chi}\end{align}
by the induction hypothesis on $p$. We estimate $P_2$. By Cauchy-Schwarz inequality applied to $(\varphi^{m-1} d \varphi) \wedge (\varphi^m d^c v_1)$ and Lemma \ref{le-obser}, we have 
	\begin{align} \label{ine-P2}
P_2^2  & \lesssim  \int_{\mathbb{B}} \chi^2 h_n^2 \varphi^{2(m-1)} d\varphi  \wedge d^c \varphi  
		\wedge R   \int_{\mathbb{B}} \chi^2 h_n^2 \varphi^{2m} dv_1  \wedge d^c v_1 		\wedge R \\
		\nonumber
		&\leq \int_{\mathbb{B}} \chi^2 h_n^2 \varphi^{2(m-1)} dd^c  \psi_n
		\wedge R \int_{\mathbb{B}} \chi^2 h_n^2  \varphi^{2m} dd^c (v^2_1) 
		\wedge R \\
		\nonumber
		&\lesssim   n I_{n,m-1,p,\supp \chi} I_{n,m,p,\chi} \lesssim n^m I_{n,m,p,\chi}
	\end{align}
	by the induction hypothesis on $m$ (for $p$ fixed).   
Finally, we want to control $P_3$. Again by the Cauchy-Schwarz inequality applied to $d h_n \wedge (h_n d^c v_1)$ and Lemma \ref{le-obser}, 
\begin{align*}
	P_3^2 &  \leq 4\int_{
		\mathbb{B}} \chi^2 \varphi^{2m}  dh_n \wedge d^c h_n 
	\wedge R  \int_{
		\mathbb{B}} \chi^2 \varphi^{2m} h_n^2 dv_1 \wedge d^c v_1 
	\wedge R  \\
	& \lesssim  I_{n,m,p, \chi} \int_{
		\mathbb{B}} \chi^2 \varphi^{2m}  T_n 
	\wedge R \le I_{n,m,p,\chi} J_{n,m,p, \supp \chi} \lesssim n^m I_{n,m,p,\chi}
\end{align*}
by Lemma \ref{le-estimateJ}. Combining the last inequality, (\ref{ine-P1}) and (\ref{ine-P2}) gives $I_{n,m,p,\chi}^2 \lesssim n^m I_{n,m,p,\chi} $ hence the desired assertion. This finishes the proof.
\end{proof}

\subsection{Capacity estimates}
For every Borel set $E \subset \B$ the relative extremal psh function with respect to $E$ is defined by 
\begin{align*}
	u_E&:= \sup\{ u \ \mathrm{psh} \ \mathrm{on} \ \mathbb{B}, \ u \leq 0 \ \mathrm{on }\ \mathbb{B}, \   u \leq -1 \ \mathrm{on }\ E \}.
\end{align*}
Let $u^*_E$ be the upper semicontinuous regularization of $u_E$. Recall that $-1 \le u^*_E \le 0$ is a psh function on $\B$, which only differs from $u_E$ on a pluripolar set. Moreover, $(dd^c u^*_E)^k$ vanishes outside $\overline E$ and if $E$ is relatively compact in $\B$ then  
\begin{align}\label{eq-caprelative}
\mathrm{Cap}(E,\B)= \int_\B (dd^c u^*_E)^k= \int_{\overline E} (dd^c u^*_E)^k,
\end{align}
see \cite{BedfordTaylor} and also \cite{Klimek}.

Fix a compact set $K \Subset \B$. Let $2<\lambda<4$ be a constant and let $\varphi$ and $\psi$ be as above.   For $n\in \N$, we consider 
\[ K_n:= \{ z \in K, \ \varphi(z) \geq 2^n \ \mathrm{and} \ \psi \geq \ -\lambda^{n} \}. \]   
 
 Note that $K_n$ is a priori not closed. On the other hand by \cite[Theorem 2.10]{DinhMarVu} or \cite[Theorem 22]{VignyW}, $\varphi$ is quasi-continuous (with respect to capacity), \emph{i.e,} for every constant $\epsilon>0$, there exists an open subset $V_\epsilon$ in $\B$ such that $\mathrm{Cap}(V_\epsilon,\B) \le \epsilon$ and $\varphi$ is continuous on $\B \backslash V_\epsilon$. Hence $K_n \backslash V_\epsilon$ is closed on $\B$ (note that the set $\{\psi \ge -\lambda^n\}$ is already closed).

\begin{corollary}\label{key} Let $\varphi \in W^*(\B)$ be a non-negative function with $\|\varphi\|_* \le 1$ as above. Let $u_n:= u_{K_n}$.
Then,	for every $m$, there exists a constant $c_m$ independent of $\varphi$ such that for all $n \in \N$
	\[ \int_{K_n} (dd^c u^*_n)^k \le \mathrm{Cap}(K_n,\B) \leq  c_m  (\lambda/4)^{mn}.\]
\end{corollary}

\begin{proof} We have 
$$\mathrm{Cap}(K_n,\B)= \int_{\overline K_n} (dd^c u^*_n)^k \ge \int_{K_n} (dd^c u^*_n)^k.$$
As noted above,  $K_n$ is not necessarily closed. Let $l$ be a positive integer and $V_l$ be an open subset in $\B$ so that $\mathrm{Cap}(V_l, \B) \le l^{-1}$ and $\varphi$ is continuous on $\B \backslash V_l$. 
 Observe that 
\begin{align}\label{ine-danhgiacapa}
\mathrm{Cap}(K_n,\B) &\le \mathrm{Cap}(K_n \backslash V_l, \B)+ \mathrm{Cap}(V_l, \B) \\
\nonumber
&\le \mathrm{Cap}(K_n \backslash V_l, \B)+l^{-1}\\
\nonumber
&= \int_{K_n \backslash V_l} (dd^c u^*_{K_n \backslash V_l})^k+ l^{-1}
\end{align}
because of (\ref{eq-caprelative}) and the fact that  $K_n \backslash V_l$ is closed (hence compact). 
%

 Observe that $h_{2\lambda^n} \geq 1/2$ on $K_n$. Thus
	\begin{align*}
	  \int_{K_n \backslash V_l} (dd^c u^*_{K_n \backslash V_l} )^k  &\leq  4 \int_{K_n \backslash V_l} h_{2\lambda^n}^2 4^{-nm}\varphi^{2m} (dd^c u^*_{K_n \backslash V_l} )^k \\
	  &\le 4 \int_{K} h_{2\lambda^n}^2 4^{-nm}\varphi^{2m} (dd^c u^*_{K_n \backslash V_l} )^k \\
&	   \le  c 4^{-nm}\lambda^{nm}  
\end{align*}
	by Lemma \ref{le-estimateI} applied to $\varphi$, where $c$ is a constant depending only on $m,K$ (note that $K_n \subset K$). This combined with (\ref{ine-danhgiacapa}) gives
$$\mathrm{Cap}(K_n,\B) \le C 4^{-nm} \lambda^{mn}+ l^{-1}$$
for some constant $C>0$ depending only on $k$ and $m$. Letting $l \to \infty$ yields	the second desired inequality. The proof is finished. 
\end{proof}

\begin{proof}[End of the proof of Theorem \ref{tm:main?}]
	Let $\alpha \in [1,2)$. We first  consider $\varphi \in W^*(\B)$ with $\varphi \ge 0$ and let $\psi$ be a negative psh function on $\B$ with $d \varphi \wedge d^c \varphi \le dd^c \psi$. Let $K_n, u_n$ be as above.  Let $\lambda$ be a number with $2^\alpha < \lambda$. Consider the non-positive function
	\[ u:= \sum_{n=1}^\infty 2^{n \alpha } \left(u^*_n + \frac{\max(\psi, -\lambda^n)}{\lambda^n}\right).\]
	We will prove that \\
	
	\noindent
	\textbf{Claim.} $2^\alpha u \leq -\varphi^\alpha$ outside a pluripolar set. \\
	
	\noindent
	Let 
	$$A_n:=\{x \in \B: \varphi(x) \in [2^n, 2^{n+1}),\quad \psi(x) >-\lambda ^n\}.$$ 
  Recall that $\varphi$ is identified with a representative which is well defined, hence finite, outside a pluripolar set that we denote by $A_\infty$.
  
  Let $x \in \B \backslash \big(\cup_{n=1}^\infty A_n \cup A_\infty\big)$. Thus,  there exists a positive integer $n$ so that  $\varphi(x) \in [2^n, 2^{n+1})$ and $\psi(x)  \le -\lambda ^n$ (note that $0 \le \varphi(x)< \infty$). Consequently
	$$u(x) \leq  2^{n \alpha } \left(\frac{\max(\psi, -\lambda^n)}{\lambda^n}\right)= -2^{n \alpha }.$$
	In other words, one obtains
	$$2^\alpha u \le -\varphi^\alpha$$
	on the complement of $A:= \cup_{n \in \N} A_n \cup A_\infty$. Since $u_n^*= u_n$ outside some pluripolar set $E_n$, we see that for $x \in A_n \backslash E_n$, there holds $u(x) \leq 2^{n \alpha } u_n(x)=- 2^{n \alpha }$, where we note that $u_n=-1$ on $K_n$, which contains $A_n$. Hence 
	$$2^\alpha u \le - \varphi^\alpha$$
	on $A \backslash (\cup_{n=1}^\infty E_n \cup A_\infty)$. Hence  the claim follows. \\

	We now show that $u$ is not identically $-\infty$. Observe first that the series 
	$$w:= \sum_{n=1}^\infty 2^{n \alpha }  \frac{\max(\psi, -\lambda^n)}{\lambda^n}$$ is a well-defined psh function by our choice of $\lambda$.

	 Let $\rho(z):=\|z\|^2-1$. Hence $\rho=0$ on $\partial \mathbb{B}$ and $dd^c \rho = \omega$. Since $\rho<0$ in $\B$ and $\rho$ is continuous, we see that there exists a constant $M_K$ depending on $K$ so that $M_K\rho \le -1$ on $K$ (hence on $K_n$ because $K_n \subset K$). Thus 
$$	u_n \ge M_K \rho
$$
on $\B$ for every $n$,  by the envelope in the definition of $u_n$. It follows that for every $w \in \partial \B$ and $n \in \N^*$, one has
\begin{align}\label{ine-chanduoiun}
	\liminf_{z \to w, z \in \B} (2^{n \alpha} u_n^*(z)- n^{-2}\rho(z)) \ge  |2^{n \alpha} M_K- n^{-2}| \lim_{z \to w}\rho(z)=0.
\end{align}
 For $n \in \N^*$, let 
	\[B_n := \left\{ 2^{n \alpha} u^*_n < \frac{1}{n^2}  \rho  \right\}. \]
By (\ref{ine-chanduoiun}), one can apply	the comparison principle (\cite[Theorem 1.16]{K-memoir}) to $2^{n \alpha} u^*_n$  and $\frac{1}{n^2}  \rho$. Thus one obtains that
	\[  \frac{1}{n^{2k}}\int_{B_n} (dd^c \rho)^k \leq  2^{kn \alpha } \int_{B_n} (dd^c u^*_n)^k \leq  2^{kn\alpha} \mathrm{Cap}(K_n) \leq  2^{nk\alpha} c_m \frac{\lambda^{nm}}{2^{2nm}} ,\]
 for any $m \in \N$,	where $c_m$ is the constant given by Corollary~\ref{key}. 
	Hence for every $n_0 \ge 1$ we get 
	\[ \sum_{n\geq n_0}  \int_{B_n} \omega^k \leq \sum_{n\geq n_0} n^{2k} 2^{nk\alpha} c_m \frac{\lambda^{nm}}{2^{2nm}}= c_m \sum_{n \ge n_0} n^{2k} (2^{k\alpha}  4^{-m}\lambda^m)^n. \]
	Thus by choosing $m$ large enough (so that $2^{k\alpha} 4^{-m} \lambda^m <1$), we see that 
	$$\sum_{n\geq n_0}  \int_{B_n} \omega^k < \int_{\mathbb{B}} \omega^k$$ for $n_0$ large enough (independent of $\varphi$ and $K$). In particular, there is $x_0 \in \mathbb{B} \backslash \cup_{n \geq n_0} B_n$. Hence we get  
	\[ \forall n\geq n_0,  \ 0 \geq 2^{n\alpha} u^*_n(x_0) > n^{-2}\rho(x_0).   \]
	It follows that 
	$$u(x_0)= w(x_0)+ \sum_{n \ge 1} 2^{n \alpha} u^*_n(x_0) \ge w(x_0)+O(1)+ \rho(x_0) \sum_{n \ge 1} n^{-2} \ge w(x_0) +O(1)+2 \rho(x_0)>- \infty.$$
	Consequently, $u \not \equiv -\infty$. This gives the existence of $u$ in Theorem~\ref{tm:main?} when $\varphi \ge 0$.\\ 
	
	In the next paragraphs, we will show that one can choose $u$ so that  the $L^1$-norm of $u$ is bounded uniformly (still for $\varphi \ge 0$). Define
	$$M_\varphi:= \inf\{ \|u\|_{L^1(K)}: |\varphi|^\alpha \le -u, \text{ $u$ is a negative psh function on $\B$}\}.$$
Let $M$ be the supremum of $M_\varphi$ for $\varphi$ running over non-negative functions on $W^*(\B)$ of $*$-norm at most $1$. 
We check that $M< \infty$.  Suppose on contrary that $M= \infty$.  	Hence, one can find a sequence $(\varphi_n)_n$ of non-negative functions in $W^*(\B)$ with $\|\varphi_n\|_* \leq 1$ such that  $M_{\varphi_n}  \geq 2^n$.  Define 
	 \[v:= \sum_{n \ge 1}n^{-2} \varphi_n\]
	 which is an element in $W^*(\B)$ because
 $W^*(\B)$ is a Banach space. Hence by the previous part of the proof, there exists a negative psh function $u$ on $\B$ with $|u| \geq v^\alpha$. In particular, $|u| \geq n^{-2 \alpha } \varphi^\alpha_n$. Thus, we get $ n^{2 \alpha } |u| \ge \varphi_n^\alpha$. It follows that 
$$M_{\varphi_n} \le \|n^{2 \alpha} u\|_{L^1(K)} = n^{2 \alpha} \|u\|_{L^1(K)}\ll 2^n$$
if $n$ is big enough. This is a contradiction because $M_{\varphi_n} \ge 2^n$. Hence $M<\infty$. We infer that for every $\alpha \in [1,2)$,  there exists a constant $C_\alpha>0$ so that for every non-negative function $\varphi\in W^*(\B)$ of $*$-norm $\le 1$, there exists a negative psh function $u$ on $\B$ satisfying $- u \ge |\varphi|^\alpha$ on $K$ and $\|u\|_{L^1(K)} \le C_\alpha$. 

Now we consider the general case where $\varphi \in W^*(\B)$ with $\|\varphi\|_* \le 1$. Hence $\varphi_1:= \max(\varphi, 0)$ and $\varphi_2:= -\min(\varphi, 0)$ are both of $*$-norm $\le 1$. Applying the above result to $\varphi_1$ and $\varphi_2$, we find  negative psh functions $u_1,u_2$ on $\B$ so that 
$$-u_j \ge |\varphi_j|^\alpha, \quad \|u_j\|_{L^1(K)} \le C_\alpha.$$
This combined with the fact that $\varphi= \varphi_1 - \varphi_2$ yields
$$|\varphi|^\alpha \lesssim -(u_1+u_2), \quad \|u_1+ u_2\|_{L^1(K)} \le 2 C_\alpha.$$
By putting $u:= u_1+ u_2$, we see that $u$ satisfies the desired properties. This finishes the proof.
\end{proof}

\begin{remark} \normalfont  Observe that if $\varphi$ satisfies $d \varphi \wedge d^c \varphi \le dd^c \psi$ with $\psi \in L^\infty$, then the proof is much simpler and one can show actually that for any $\alpha \in \R^+$, there is a psh function $u$ such that $|u| \geq |\varphi|^\alpha$ on $K$ and the $L^1$-norm on $K$ of $u$ is uniformly bounded.
\end{remark}

Let $X$ be a compact K\"ahler manifold and $\omega$ be a K\"ahler form on $X$. We can define $W^*(X)$ in a way similar to $W^*(U)$, see \cite{DinhSibonyW,VignyW}. Recall that for a K\"ahler form $\eta$ on $X$, we say that a function $u$ is  $\eta$-psh if $dd^c u + \eta \geq 0$. Here is a global version of  Theorem \ref{tm:main?}. 

\begin{theorem} \label{th-tm:main?global} Let $\alpha \in [1,2)$. Then there exists a constant $C>0$ such that for every $\varphi \in W^*(X)$ with $\|\varphi\|_* \le 1$ there is a negative $C \omega$-psh function $u$ on $X$ such that 
$$|\varphi|^\alpha  \le -u, \quad \|u\|_{L^1(X)} \le C.$$
\end{theorem}

\proof We follow almost line by line the proof of Theorem \ref{tm:main?}. The only new issue to handle is to choose $u_n$ more carefully to obtain that $u$ is $C \omega$-psh for some uniform constant $C$. Recall that for Borel set $E$ in $X$, 
$$\mathrm{Cap}_\omega(E):= \sup \big\{ \int_X (dd^c v+ \omega)^n: 0 \le v \le 1,  \text{ $v$ is $\omega$-psh}\big\}.$$
For every K\"ahler form $\eta$ on $X$, we put 
\begin{align*}
	u_{E, \eta}:= \sup\{ u  \text{ negative $\eta$-psh }:     u \leq -1 \ \mathrm{on }\ E \}.
\end{align*}
Let $u^*_{E,\eta}$ be the upper semicontinuous regularization of $u_E$. As in the local setting, one has $-1 \le u^*_{E,\eta} \le 0$ which is $\eta$-psh on $X$, and $u^*_{E,\eta}$ differs only from $u_{E,\eta}$ on a pluripolar set and $(dd^c u^*_E+ \eta)^k$ vanishes on $ \{u^*_{E, \eta} <0\} \backslash \overline E$ (when $\{u^*_{E, \eta} <0\}$ is non empty), and
\begin{align}\label{eq-caprelative_compact}
\mathrm{Cap}_\eta(E)= \int_X -u^*_{E, \eta}(dd^c u^*_{E,\eta}+ \eta)^k= \int_{\overline E} -u^*_{E, \eta} (dd^c u^*_{E,\eta}+ \eta)^k.
\end{align}
We refer to \cite{GZ} for proofs of these statements which follow, more or less, from those in the local setting in \cite{BedfordTaylor}.

Let $\varphi \in W^*(X)$ with $\|\varphi\|_* \le 1$ and let $T$ be a closed positive $(1,1)$-current on $X$ so that $d \varphi \wedge d^c \varphi \le T$ and $\int_X T \wedge \omega^{k-1} \le 1$.  Because of the bound on the mass of $T$, there exists a constant $C$ independent of $T$ so that we can write  $T= dd^c \psi + \theta$ for some smooth form $\theta$ and $\theta$-psh function $\psi$ with $\sup_X \psi =0$ and $\theta \le C \omega$.    As in the proof of Theorem \ref{tm:main?}, it suffices to consider the case where $\varphi \ge 0$. 

Let $2<\lambda<4$ be a constant $2^\alpha < \lambda$.     For $n\in \N$, we consider 
\[ K_n:= \{ z \in X, \ \varphi(z) \geq 2^n \ \mathrm{and} \ \psi \geq \ -\lambda^{n} \} \]   
and 
$$u_n:= u_{K_n, 3^{-n \alpha}\omega}.$$
Arguing as in the proof of Corollary \ref{key}, one sees that 
	for every $m$, there exists a constant $c_m$ independent of $\varphi$ such that for all $n \in \N$
	\begin{align}\label{ine-danhgiacapglobal}
	 \mathrm{Cap}_{3^{-n \alpha} \omega}(K_n) \leq  c_m  (\lambda/4)^{mn}.
	 \end{align}
 Consider the non-positive function
	\[ u:= \sum_{n=1}^\infty 2^{n \alpha } \left(u^*_n + \frac{\max(\psi, -\lambda^n)}{\lambda^n}\right).\]
As before we have  $2^\alpha u \leq -\varphi^\alpha$ outside a pluripolar set. Let 
$$\eta:= \sum_{n=1}^\infty 2^{n \alpha} (3^{-n \alpha} \omega)+ 2^{n \alpha}\lambda^{-n} C \omega \le C' \omega$$
for some constant $C'>0$ (depending only on $\alpha, \lambda$) by our choice of $\lambda$.  Since $u^*_n$ is $3^{-n \alpha} \omega$-psh and $\psi$ is $C \omega$-psh, we get
$$dd^c u + \eta \ge 0 $$
if $u \not  \equiv -\infty$. 
It follows that $u$ is $C' \omega$-psh if $u \not \equiv -\infty$. 	It remains to check that $u$ is not identically equal to $-\infty$. To this end, we argue as in the proof of Theorem \ref{tm:main?}. 

For $n \in \N^*$, let 
	\[A_n := \left\{ 2^{n \alpha} u^*_n < -\frac{1}{n^2}  \right\}. \]
	By the comparison principle (see \cite[Theorem 6.4]{K-memoir}), for any $m \in \N$,
	\begin{align*}
	  \int_{A_n} (3^{- n \alpha} \omega)^k &\leq  \int_{A_n} (dd^c u^*_n+ 3^{-n \alpha} \omega)^k  \\
	  &\le  2^{n \alpha} n^{2} \int_{A_n} -u^*_n (dd^c u^*_n+ 3^{-n \alpha} \omega)^k \\
	  & \le  2^{n \alpha} n^{2} \mathrm{Cap}_{3^{- n \alpha} \omega}(K_n) \leq  c_m 2^{n \alpha} n^{2} \frac{\lambda^{nm}}{2^{2nm}},
	  \end{align*}
	  by (\ref{ine-danhgiacapglobal}).  	Hence for every $n_0 \ge 1$ we get 
	\[ \sum_{n\geq n_0}  \int_{A_n} \omega^k \leq \sum_{n\geq n_0} c_m n^2  3^{n(k+1) \alpha} \frac{\lambda^{nm}}{2^{2nm}}= c_m \sum_{n \ge n_0} n^{2} (3^{(k+1)\alpha}  4^{-m}\lambda^m)^n. \]
	Thus by choosing $m$ large enough (so that $3^{(k+1)\alpha} 4^{-m} \lambda^m <1$), we see that 
	$$\sum_{n\geq n_0}  \int_{A_n} \omega^k < \int_{X} \omega^k$$ for $n_0$ large enough (independent of $\varphi$). In particular, there is $x_0 \in X \backslash \cup_{n \geq n_0} A_n$. Hence we get  
	\[ \forall n\geq n_0,  \ 0 \geq 2^{n\alpha} u^*_n(x_0) > n^{-2}.   \]
Hence $u(x_0)> -\infty$. The fact that one can choose $u$ so that its $L^1$-norm is uniformly bounded is proved exactly in the same way as in the proof of Theorem \ref{tm:main?}.  This finishes the proof.
\endproof

\begin{example} \label{ex-alpha2} \normalfont Let $\alpha>2$ and  $k=1$. For every $\delta \in (0,1/2)$ consider $\varphi(z):= (-\log |z|^2)^{1/2- \delta}$. Direct computations show that
$$d \varphi \wedge d^c \varphi=  \frac{i d z \wedge d \bar z}{ \pi |z|^2 (\log |z|^2)^{1+2\delta}}$$
which is of finite mass on $U:= B(0,1/2)$ in $\C$. Hence $\varphi \in W^{1,2}(U)= W^*(U)$. However since $\alpha>2$, one can choose $\delta$ small enough so that  $\beta:= \alpha (1/2-\delta)>1$. Consequently  $\varphi^\alpha= (-\log |z|^2)^{\beta}$ is not bounded from above by minus of a subharmonic function on $B(0,1/4)$ because if there were such a function, then its Lelong number at $0$ would be equal to $\infty$ (a contradiction). 
\end{example}

\section{Proof of Theorem \ref{th:Lebesgues}} \label{sec-mainthoerporof}

We first present some more auxiliary results about the plurifine topology. Although the materials seem to be standard, we will give details here because we could not find a proper reference for them.  

\begin{lemma}\label{le-plurifinetopo} Let $u$ be a negative psh function on $\B$. Let $x_0 \in \B$ so that $u(x_0)>- \infty$. Let $\delta>0$ be a constant and let 
$$E_\delta(u):= \{x \in U: |u(x)-u(x_0)| \ge \delta \}.$$
  Let $c_r:= \mathrm{Leb}(B(x_0,r))$ and $b_r:= \mathrm{Leb}(B(x_0,r) \cap E_\delta)$.
Then $b_r/c_r \to 0$ as $r \to 0$.
\end{lemma}

\proof  By upper semicontinuity of $u$, for every constant $\epsilon>0$, one gets 
$$u(x) \le u(x_0) + \epsilon$$
for $x \in B(x,r)$ with $r<r_\epsilon$ small enough. Hence for $\epsilon<\delta$ and  $r<r_\epsilon$ and  $x \in E_\delta \cap B(x,r)$ there holds $u(x) \le u(x_0) -\delta$. This combined with the submean inequality gives
\begin{align*}
c_r u(x_0) &\le  \int_{B(x_0,r)} u d \mathrm{Leb} \\
&\le  \int_{B(x_0,r) \backslash E_\delta} u d \mathrm{Leb}+ \int_{B(x_0,r) \cap E_\delta} u d \mathrm{Leb} \\
&\le \big(u(x_0)+\epsilon\big) (c_r- b_r)+ \big(u(x_0)-\delta\big) b_r\\
& \le c_r u(x_0)+  \epsilon c_r - \delta b_r.
\end{align*}
Dividing both sides by $\delta c_r$ we obtain
$$b_r/c_r \le \epsilon/ \delta.$$
Hence $\limsup_{r \to 0} b_r/c_r \le \epsilon/\delta$ for every $\epsilon < \delta$.
Letting $\epsilon \to 0$ one gets $b_r/c_r \to 0$ as desired. 
\endproof

\begin{lemma}\label{le-plurifinetopo2}
Let $u_1,\ldots, u_m$ be negative psh functions on $\B$ and $x_0 \in \B$ so that $u_j(x_0)> -\infty$ for $1 \le j \le m$. Let $\delta>0$ be a constant and  $E_\delta:= \cup_{j=1}^m E_\delta(u_j)$.  Let $c_r:= \mathrm{Leb}(B(x_0,r))$. Then  for every negative psh function $u'$ on $U$ with $u'(x_0)> -\infty$ there holds
$$c_r^{-1}\int_{B(x_0,r)  \cap E_\delta} |u'| d \mathrm{Leb} \to 0$$ 
and
$$c_r^{-1}\int_{B(x_0,r)  \backslash E_\delta} u' d \mathrm{Leb} \to u'(x_0)$$ 
as $ r \to 0$.
\end{lemma}

\proof The second desired convergence is a direct consequence of the first one. We prove now the first desired assertion.  We use again  the notation of the previous lemma $b_r= \mathrm{Leb}(B(x_0,r) \cap E_\delta)$. Since $u_j(x_0)>-\infty$, using Lemma  \ref{le-plurifinetopo}, we infer that 
\begin{align}\label{conver-brar}
b_r/c_r \to 0
\end{align}
as $r \to 0$. By upper semicontinuity of $u'$, for every constant $\epsilon>0$, one gets 
$$u'(x) \le u'(x_0) + \epsilon$$
for $x \in B(x,r)$ with $r<r'_\epsilon$ small enough. By definition of  $E_\delta$, we have
\begin{align*}
c_r^{-1}\int_{B(x_0,r)} u' d \mathrm{Leb} &=  c_r^{-1} \int_{B(x_0,r) \backslash E_\delta} u' d \mathrm{Leb}+ c_r^{-1} \int_{B(x_0,r) \cap E_\delta} u' d \mathrm{Leb} \\
&\le  \big(u'(x_0)+ \epsilon\big)(1- b_r/c_r)+ c_r^{-1}\int_{B(x_0,r) \cap E_\delta} u' d \mathrm{Leb}
\end{align*}
if $r < r'_\epsilon$. Letting $r \to 0$ and using (\ref{conver-brar}) we get
$$u'(x_0) \le u'(x_0)+\epsilon+ \liminf_{r \to 0}c_r^{-1}\int_{B(x_0,r) \cap E_\delta} u' d \mathrm{Leb}.$$
It follows that 
$$\liminf_{r \to 0}c_r^{-1}\int_{B(x_0,r) \cap E_\delta} u' d \mathrm{Leb} \ge -\epsilon$$
for every $\epsilon>0$.  The second desired inequality hence follows because $u' \le 0$. The proof is complete.
\endproof

We recall that the plurifine topology on $\B$ is the coarsest topology that makes psh functions on $\B$ continuous (\cite{BT_fine_87}).  Intersection of finitely many sets of type $\{v>0\}$ or $\{v<0\}$ for psh functions $v$ on some open subset in $\B$ form a basis of this topology. 

\begin{lemma}\label{le-basisplurifine} The family of sets of form $V \cap \cap_{j=1}^m \{x: |u_j(x)- u_j(x_0)|< \delta\}$, where $V$ runs over open subsets in $\B$ with respect to the Euclidean topology, $u_j$ bounded psh functions on $V$, and $x_0$ runs over points in $V$, is a basis of the plurifine topology. 
\end{lemma}

\proof
Since  the set $\{v<0\}$ for every psh function $v$ is open in the Euclidean topology, we see that the family of sets of form $V':= V \cap \cap_{j=1}^m \{v_j >0\}$, where $V$ is an open set in the Euclidean topology, $m \in \N$ and $v_j$ psh functions on $V$, is a basis of plurifine topology. It suffices to consider only bounded psh functions by replacing $v_j$ by $\max\{v_j, 0\}$ in the definition of elements in the latter basis. Let $x_0 \in V'$. Since $v_j(x_0)>0$, it is clear that one can find a constant $\delta>0$ such that $$ x_0 \in V \cap \cap_{j=1}^m \{x: |v_j(x)- v_j(x_0)|< \delta\} \subset V'.$$
This finishes the proof. 
\endproof


\begin{proposition}\label{pro-plurifinetopo3}
Let $V$ be a non-empty plurifinely open subset in $\B$. Let $x_0 \in V$ and $u$ be a psh function on $\B$ with $u(x_0)> -\infty$.   Let $c_r:= \mathrm{Leb}(B(x_0,r))$. Then we have 
$$c_r^{-1}\int_{B(x_0,r)  \backslash V} |u| d \mathrm{Leb} \to 0$$ 
and 
$$c_r^{-1}\int_{B(x_0,r)  \cap V} |u- u(x_0)| d \mathrm{Leb} \to 0$$ 
as $ r \to 0$.
\end{proposition}

\proof By Lemma \ref{le-basisplurifine}, it is enough to prove the desired assertions for $V=\cap_{j=1}^m \{x: |u_j(x)- u_j(x_0)|< \delta\}$, where $u_j$ bounded psh functions on some open subset $V'$. The desired assertions now follow from Lemma \ref{le-plurifinetopo2}.
\endproof

A real function $f$ is plurifinely continuous in a plurifinely open set $W$ if and only if for every open interval $I \subset \R$ and for every $x \in W$, there exists a plurifinely open set $B$ containing $x$ such that $f(B) \subset I$. The following result is the second ingredient in the proof of Theorem \ref{th:Lebesgues}.

\begin{proposition}\label{pro-plurifinetopoW*} (\cite[Theorem 22]{VignyW}) Every function $\varphi$ in $W^*(\B)$ is plurifinely continuous outside some pluripolar set. Precisely, there exists a family $\mathcal{F}$ of bounded psh functions on $\B$ such that 
$$E:= \cap_{v \in \mathcal{F}}\{v \le -1\}$$
is pluripolar (and is closed in the plurifine topology) and for every $x \in \B \backslash E$ and $v \in \mathcal{F}$ with $x \in \{v >-1\}$, one has that   $\varphi(x') \to \varphi(x)$ as $x' \to x$ and $x'$ remains in $\{v >-1\}$.  
\end{proposition}

 We always identify $\varphi$ with a fixed representative of $\varphi$ (see Introduction for references). Thus, to be precise, the statement of Proposition \ref{pro-plurifinetopoW*} means that if $\varphi'$ is a representative of $\varphi$, then $\varphi'$ is plurifinely continuous outside some pluripolar set.

\proof We recall the proof for readers' convenience.  By \cite[Theorem 2.10]{DinhMarVu} or \cite{VignyW}, we know that for every constant $\epsilon>0$, there exists an open subset $Y_\epsilon$ in $\B$ so that $\mathrm{Cap}(Y_\epsilon, \B) \le \epsilon$ and $\varphi$ is continuous on $\B \backslash Y_\epsilon$. Let $(U_j)_{j \ge 1}$ be an increasing sequence of relatively open subsets in $\B$ so that and $\overline U_j \subset U_{j+1}$ and $\B= \cup_{j \ge 1} U_j$. Put $K_j:= Y_{1/j}  \cap U_j$ which is relatively compact in $\B$. Hence $\varphi$ is continuous on $U_j \backslash K_j$.  We have 
$$\mathrm{Cap}(K_j, \B) \le \mathrm{Cap}(Y_{1/j}, \B) \le 1/j$$
for every $j$.  Let $\mathcal{F}_j$ be the set of psh functions $u$ on $\B$ such that $-1\le u \le 0$ on $\B$ and $u = -1$ on $K_j$. 
Let 
$$u_j^*= u_{K_j}^*:= \big(\sup \{u \text{ psh  on }\B: u \le 0 \text{ on } \B, \quad u \le -1 \text{ on } K_j\}\big)^*,$$
By considering $\max\{u,-1\}$ instead of $u$ in the envelope defining $u_j$, one sees that 
$$u_j^*:= \big(\sup \{u: u \in \mathcal{F}_j\}\big)^*,$$
We define 
$$K'_j:= \cap_{u \in \mathcal{F}_j} \{u \le -1\}$$
which contains $K_j$. Note that $K'_j$ is relatively compact in $\B$ because  $\max\{M\rho, -1\} \in \mathcal{F}_j$ for $M$ large enough (recall $\rho= \|x\|^2-1$ and $K_j \Subset \B$) and $K'_j$ is plurifinely closed by definition of the plurifine topology.  Observe that if $v$ is a negative psh function on $\B$ so that $v \le -1$ on $K_j$ then $v \le -1$ on $K'_j$ because $\max\{v,-1\} \in \mathcal{F}_j$. We infer that 
$$u_j^*= u_{K'_j}^*.$$  
This combined with the fact that both $K_j$ and $K'_j$ are  relatively compact Borel sets in $\B$ gives 
$$\mathrm{Cap}(K_j, \B)= \int_\B (dd^c u^*_j)^k= \int_\B (dd^c u_{K'_j}^*)^k =\mathrm{Cap}(K'_j, \B).$$
We thus obtain that $\varphi$ is plurifinely continuous on $U_j \backslash K'_j$ and $\mathrm{Cap}(K'_j, \B) \le 1/j$. Let 
$$E:= \cap_{j \ge 1} K'_j= \cap_{u \in \mathcal{F}} \{u \le -1\},$$
where $\mathcal{F}:= \cup_{j \ge 1}\mathcal{F}_j$ which is a family of psh functions $-1 \le v \le 0$ on $\B$. Hence $\mathrm{Cap}(E,\B) \le \mathrm{Cap}(K'_j,\B) \le 1/j$ for every $j$. It follows that
$$\mathrm{Cap}(E, \B)=0,$$
in other words, $E$ is pluripolar (note $E$ is Borel). Since $K'_j$ is plurifinely closed, we see that $\B \backslash K'_j$ is plurifinely open. We don't need this observation for the rest of the proof. 

We note that $K_j \subset K'_j$. Let $x \in \B \backslash E$. Hence there are $j \in \N$ and $v \in \mathcal{F}_j$ so that 
$$x \in \{v >-1\} \cap U_j \subset U_j \backslash K'_j.$$
 This combined with the continuity of $\varphi $ on $U_j \backslash K'_j$ gives 
$\varphi(x') \to \varphi(x)$ as $x' \to x$ and $x' \in U_j \backslash K'_j$ (in particular it holds when $x'$ remains in $\{v>-1\}\cap U_j$). 
This finishes the proof. 
\endproof

Let $(\mu_{\epsilon})_{\epsilon \in (0,1]}$ be a sequence of probability measures on $\C^k$ such that $\supp \mu_{\epsilon} \subset B(0,\epsilon)$ and $$\epsilon^{2k} \mu_\epsilon \le M \mathrm{Leb},$$ for some constant $M>0$ independent of $\epsilon$. Observe that $\mu_\epsilon$ converges weakly to  $\delta_0$ as $\epsilon \to 0$.  We say that such a sequence is \emph{an approximation of unity}.  Two important examples of approximations of unity are 
$$\mu_{1,\epsilon}:= \frac{1}{\mathrm{Leb}(B(0,\epsilon))} \bold{1}_{B(0,\epsilon)} \mathrm{Leb}, \quad \mu_{2, \epsilon}:= \epsilon^{-2k}\chi(x/ \epsilon) \mathrm{Leb},$$
 where $\chi$ is a radial cut-off function as above. We recall the following standard fact.

\begin{lemma} \label{le-lebespoint} Let $\varphi \in L^1_{loc}(U)$ and $x \in U$ be a point such that
	$$\lim_{\epsilon \to 0} \frac{1}{\mathrm{Leb}(B(x,\epsilon))} \int_{B(x,\epsilon)} |\varphi- \varphi(x)| d \mathrm{Leb}=0.$$
	Then one has 
	$$\lim_{\epsilon \to 0} \int_{\C^k} |\varphi(x+y)- \varphi(x)| d \mu_\epsilon(y)=0,$$
	for every approximation of unity $(\mu_\epsilon)_\epsilon$.
\end{lemma}

\proof Recall that $\supp \mu_\epsilon \subset B(0,\epsilon)$ and  there exists a constant $M>0$ such that $\epsilon^{2k} \mu_\epsilon \le M \mathrm{Leb}$. It follows that
$$\int_{\C^k} |\varphi(x+y)- \varphi(x)| d \mu_\epsilon(y) \lesssim \epsilon^{-2k} \int_{B(x, \epsilon)}|\varphi - \varphi(x)| d \mathrm{Leb}$$
which converges to $0$ as $\epsilon \to 0$ by the hypothesis. The desired convergence follows. This finishes the proof.
\endproof

\begin{proof}[End of the proof of Theorem \ref{th:Lebesgues}]Let us now finish the proof of Theorem \ref{th:Lebesgues}. It suffices to work locally and we can assume that $U= \B$ so by Theorem \ref{tm:main?} there exists  a psh function $u$ on $\B$ such that  
	\begin{align}	\label{ine-chantrenvarphi}	
		|\varphi|\leq -u
	\end{align}
	on $\B$.   By Proposition \ref{pro-plurifinetopoW*}, there exists a plurifinely closed pluripolar set $E$ such that $u$ is locally finite outside $E$ and $\varphi$ is  plurifinely continuous outside $E$. Let $x_0 \not \in E$.  
Let $V$ be a plurifinely open neighborhood of $x_0$ in $\B \backslash E$.  Let $c_r:= \mathrm{Leb}(B(x_0,r))$.  	Using (\ref{ine-chantrenvarphi}),  one obtains
\begin{align*}
c_r^{-1}\int_{B(x_0,r)} |\varphi- \varphi(x_0)| d \mathrm{Leb}&=
c_r^{-1}\int_{B(x_0,r)  \cap V} |\varphi- \varphi(x_0)| d \mathrm{Leb}+c_r^{-1}\int_{B(x_0,r)  \backslash V} |\varphi- \varphi(x_0)| d \mathrm{Leb} \\
& \le c_r^{-1}\int_{B(x_0,r)  \cap V} |\varphi- \varphi(x_0)| d \mathrm{Leb}+c_r^{-1}\int_{B(x_0,r)  \backslash V} (|u|+|u(x_0)|) d \mathrm{Leb}.
\end{align*}	
Let $I_1(r),I_2(r)$ be the first and second terms in the right-hand side of the last inequality. Since  $\varphi$ is plurifinely  continuous at $x_0$, one sees that for every constant $\epsilon>0$, there exists a constant $r_\epsilon>0$ such that  $|\varphi- \varphi(x_0)|\le \epsilon$ on $B(x_0,r) \cap V$ for $r< r_\epsilon$. It follows that $I_1(r) \le \epsilon$ for $r< r_\epsilon$. Hence $\lim_{r \to 0} I_1(r) =0$.   
On the other hand, the term $I_2(r)$ also tends to $0$ as $r \to 0$ by Proposition \ref{pro-plurifinetopo3} applied to $u$ and $u(x_0)$. Hence $x_0$ is a Lebesgue point of $\varphi$. Hence the complement of Lebesgue points of $\varphi$ is contained in the pluripolar set $E$. The second desired assertion follows from this and Lemma \ref{le-lebespoint}. The proof is finished. 
\end{proof}

\begin{remark} \normalfont We refer the reader to a related work \cite{SIR}, where Lebesgue sets were studied: in the article, the authors establish the fact that every bounded subharmonic function in a domain in $\C$, restricted to any real line, possesses the Lebesgue property at each point.   
\end{remark}


\begin{thebibliography}{DKW21}
	
	\bibitem[BD22]{bianchi2022equilibrium}
	F. Bianchi and T.-C. Dinh.
	\newblock Equilibrium states of endomorphisms of $\mathbb{P}^k$ II: spectral
	stability and limit theorems, 2022.
	
	\bibitem[BT82]{BedfordTaylor}
	E. Bedford and B.~A. Taylor.
	\newblock A new capacity for plurisubharmonic functions.
	\newblock {\em Acta Math.}, 149(1-2):1--40, 1982.
	
	\bibitem[BT87]{BT_fine_87}
	E. Bedford and B.~A. Taylor.
	\newblock Fine topology, \v{S}ilov boundary, and {$(dd^c)^n$}.
	\newblock {\em J. Funct. Anal.}, 72(2):225--251, 1987.
	
	

	\bibitem[DKW21]{DinhKaufmannWu}
	T.-C. Dinh, L. Kaufmann, and H. Wu.
	\newblock Products of random matrices: a dynamical point of view.
	\newblock {\em Pure Appl. Math. Q.}, 17(3):933--969, 2021.
		
	
	\bibitem[DKN20]{DinhKoloNguyen}
	T.-C. {Dinh}, S. {Ko{\l}odziej}, and N.~C. {Nguyen}.
	\newblock {The Complex Sobolev Space and H{\"o}lder continuous solutions to
		Monge-Amp{\`e}re equations}.
	\newblock {\em Bull. Lond. Math. Soc.}, 54(2022), no.2, 772--790.
	
	\bibitem[DMV20]{DinhMarVu}
	T.-C. {Dinh}, G. {Marinescu}, and D.-V. {Vu}.
	\newblock {Moser-Trudinger inequalities and complex Monge-Amp\`ere equation}.
	\newblock {\em Ann. Sc. Norm. Super. Pisa Cl. Sci.} (5)24(2023), no.2, 927--954.
	
	\bibitem[DS06a]{DinhSibonyW}
	T.-C. Dinh and N. Sibony.
	\newblock Decay of correlations and the central limit theorem for meromorphic
	maps.
	\newblock {\em Comm. Pure Appl. Math.}, 59(5):754--768, 2006.
	
	\bibitem[DS06b]{DS_tm}
	T.-C. Dinh and N. Sibony.
	\newblock Distribution des valeurs de transformations m\'eromorphes et
	applications.
	\newblock {\em Comment. Math. Helv.}, 81(1):221--258, 2006.
	
	\bibitem[GZ05]{GZ}
	V. Guedj and A. Zeriahi.
	\newblock Intrinsic capacities on compact {K}\"{a}hler manifolds.
	\newblock {\em J. Geom. Anal.}, 15(4):607--639, 2005.
	
	\bibitem[JN61]{John_nirenberg}
	F.~John and L.~Nirenberg.
	\newblock On functions of bounded mean oscillation.
	\newblock {\em Comm. Pure Appl. Math.}, 14:415--426, 1961.
	
	
	\bibitem[Kli91]{Klimek}
	M. Klimek.
	\newblock {\em Pluripotential theory}, volume~6 of {\em London Mathematical
		Society Monographs. New Series}.
	\newblock The Clarendon Press Oxford University Press, New York, 1991.
	\newblock Oxford Science Publications.
	
	\bibitem[Ko05]{K-memoir}
	S. {Ko{\l}odziej}.
	\newblock The complex Monge-Amp\`ere equation and pluripotential theory. \newblock {\em Mem. Amer. Math. Soc.}, 178(2005), no.840, x+64 pp.
	
%
%
%
		\bibitem[SIR14]{SIR}
	A.S. Sadullaev, and S.A. Imomkulov, and K. Kh. Rakhimov,
	\newblock Bounded subharmonic functions possess the {L}ebesgue property
		at each point.
	\newblock {\em Translation of Mat. Zametki}, 96 (6), 921--925, 2014.
	

	
	\bibitem[Vig07]{VignyW}
	G.~Vigny.
	\newblock Dirichlet-like space and capacity in complex analysis in several
	variables.
	\newblock {\em J. Funct. Anal.}, 252(1):247--277, 2007.
	
	\bibitem[Vig15]{Vignydecay}
	G. Vigny.
	\newblock Exponential decay of correlations for generic regular birational maps
	of {$\Bbb{P}^k$}.
	\newblock {\em Math. Ann.}, 362(3-4):1033--1054, 2015.
	
	\bibitem[Vu19]{Vu_pluripolar}
	D.-V. Vu.
	\newblock Locally pluripolar sets are pluripolar.
	\newblock {\em Internat. J. Math.}, 30(13), 2019.
	
	\bibitem[Vu20]{Vu_equilibrium}
	D.-V. Vu.
	\newblock Equilibrium measures of meromorphic self-maps on non-{K}\"{a}hler
	manifolds.
	\newblock {\em Trans. Amer. Math. Soc.}, 373(3):2229--2250, 2020.
	
\bibitem[Vu23]{Vu-log-diameter}
 D.-V. Vu.
 \newblock Continuity of functions in complex {S}obolev spaces.
 \newblock arXiv:2312.01635, 2023.	
 
\end{thebibliography}
 \end{document}